\documentclass[12pt]{amsart}
\usepackage{amssymb}
\usepackage{times}
\usepackage{graphicx}
\usepackage[all]{xypic}
\usepackage{enumerate}

\theoremstyle{plain}
\newtheorem{theorem}{Theorem}
\newtheorem{corollary}{Corollary}
\newtheorem{lemma}{Lemma}
\newtheorem{proposition}{Proposition}
\theoremstyle{definition}
\newtheorem{remark}{Remark}
\numberwithin{equation}{section}

\newcommand{\calM}{\mathcal{M}}

\newcommand{\bbF}{\mathbb{F}}
\newcommand{\bbC}{\mathbb{C}}

\newcommand{\bbP}{\mathbb{P}}
\newcommand{\bbQ}{\mathbb{Q}}

\newcommand{\bbZ}{\mathbb{Z}}

\newcommand{\lra}{\longrightarrow}

\def\Hom{{\text{Hom}}}

\def\Pic{{\text{Pic}}}

\def\O{{\text{O}}}

\begin{document}
\title [Curves of genus 6] {The moduli of curves of genus 6 and K3 
surfaces}
\author{Michela Artebani}
\address{Departamento de Matem\'atica, Universidad de Concepci\'on, Casilla 160-C, Concepci\'on, Chile}
\email{martebani@udec.cl }
\thanks{The first author was supported by: PRIN 2005: \emph{Spazi di moduli e teoria di Lie}, Indam (GNSAGA)}
\author{ Shigeyuki Kond{$\bar{\rm o}$}}
\address{Graduate School of Mathematics, Nagoya University, Nagoya,
464-8602, Japan}
\email{kondo@math.nagoya-u.ac.jp}
\thanks{The second  author is partially supported by
Grant-in-Aid for Scientific Research A-18204001 and Houga-20654001, Japan}

\begin{abstract}
We prove that the coarse moduli space of curves of genus $6$ is birational to an arithmetic quotient of a bounded symmetric 
domain of type IV by giving a period map to the moduli space of some lattice-polarized K3 surfaces.
\end{abstract}
\maketitle
\section*{Introduction}
This paper gives a birational period map between the coarse moduli space of curves of  genus six and the moduli space of 
some lattice-polarized K3 surfaces.
This kind of correspondence was given by the second author for curves of genus 3  and genus 4 in \cite{Kon1} and \cite{Kon2}. 
A part of the results in this paper was announced in \cite{Kon2}.  

Let $C$ be a general curve of genus six, then its canonical model is a quadratic section of a unique quintic Del Pezzo surface $Y\subset \mathbb P^ 5$ (e.g. \cite{SB}). The double cover of $Y$ branched along $C$ is a K3 surface $X$.
By taking the period point of $X$ we define a period map $\mathcal P$ from an open dense subset of the coarse moduli space $\mathcal M_6$ of curves of genus six to an arithmetic quotient  of a bounded symmetric domain $\mathcal D$ of type IV 
$$\mathcal P: \mathcal M_6\dashrightarrow \mathcal D/\Gamma.$$
The same construction defines rational period maps 
$$\mathcal P^*:\mathcal W_6^2\dashrightarrow \mathcal D/\Gamma^*, \ \hspace{0.4cm} \mathcal P^{**}:\mathcal {\widetilde M}_6\dashrightarrow \mathcal D/\Gamma^{**}.$$
Here the moduli space $\mathcal W_6^2$ parametrizes pairs $(C,D)$ where $C$ is a curve of genus six and $D$ is a $g_2^6$ on $C$, while $\mathcal {\widetilde M}_6$ is the moduli space of plane sextics with four ordered nodes. The group
$\Gamma^*$ is a subgroup of $\Gamma$ of index $5$ and $\Gamma^{**}$ is a normal subgroup of $\Gamma$ with  $\Gamma/\Gamma^{**} \cong S_5$.
In this paper we prove that $\mathcal P,\mathcal P^*, \mathcal P^{**}$ are birational maps and we study their behaviour both generically and at the boundary.

In the first section we review some classical properties of curves of genus six, in particular we recall the structure of the space $\mathcal W_6^2$. The natural projection map $\mathcal W_6^2\rightarrow \mathcal M_6$ is surjective by Brill-Noether theory. Its fiber over the general curve $C$ of genus six is a finite set of cardinality $5$ and any of its points gives a birational map from $C$ to a plane sextic with $4$ nodes. 
The fiber is known to be positive dimensional if and only if the curve of genus six is \emph{special}, i.e. it is either trigonal, hyperelliptic, bi-elliptic or isomorphic to a plane quintic curve.

In section 2 we define the period maps $\mathcal P, \mathcal P^*, \mathcal P^{**}$. In fact, we show that 
the map $\mathcal P^{**}$ is equivariant with respect to the natural actions of $S_5$ and the maps $\mathcal P, \mathcal P^*$ are obtained by taking the quotient for the action of  subgroups.
Afterwards, we prove that $\mathcal P,\mathcal P^*, \mathcal P^{**}$ are birational maps, in particular $\mathcal P$ induces an isomorphism  
$$\mathcal M_6\setminus \{\mbox{special curves}\}\cong (\mathcal D \setminus\mathcal H)/\Gamma$$
where $\mathcal H$ is a divisor defined by hyperplane sections associated to $(-2)$-vectors, called \emph{discriminant divisor}.

In section 3 we study the discriminant divisor $\mathcal H$ and its geometric meaning. We prove that $\mathcal H/\Gamma$ has 3 irreducible components which parametrize respectively curves of genus six with a node, pairs $(C,L)$ where $C$ is a plane quintic and $L$ is a line and pairs $(C,D)$ where $C$ is a trigonal curve of genus six and $D\in |K_C-2g_3^1|$.

In the final section we determine the structure of the boundary of the Satake-Baily-Borel compactification of $D/\Gamma$  
and we compare this compactification with the GIT compactification of the space of plane sextics.\\

\noindent \emph{Acknowledgments.} The first author would like to thank R. Laza for several helpful discussions.\\

\noindent \emph{Notation.} A lattice $L$ is a free abelian group of finite rank equipped with a non degenerate bilinear form, which will be denoted by $(\,,\,)$. \\
\noindent - The discriminant group of $L$ is the finite abelian group $A_L=L^*/L$, where $L^*=\Hom(L,\bbZ)$, equipped with the quadratic form $q_L:A_L\rightarrow \bbQ/2\bbZ$ defined by
$q_L(x+L) = (x,x) \mod 2\bbZ$.\\
- $\O(L)$ and $\O(q_L)$ will denote the groups of isometries of $L$ and $A_L$ respectively.\\ 
- A lattice is unimodular if $|A_L|=|\det L|=1$.\\
- If $M$ is the orthogonal complement of $L$ in a unimodular lattice, then $A_L\cong A_M$ and $q_M=-q_L$.\\
- We will denote by $U$ the hyperbolic plane and by $A_n, D_n, E_n$ the negative definite lattices of rank $n$ associated to the Dynkin's diagrams of the corresponding types. \\
- The lattice $L(\alpha)$ is obtained multiplying by $\alpha$ the form on $L$.\\
- The lattice $L^m$ is the orthogonal direct sum of $m$ copies of the lattice $L$.\\

\noindent We will refer the reader to \cite{N1} for basic facts about lattices.

\section{Curves of genus six and quintic Del Pezzo surfaces}
We start recalling some well-known properties of curves of genus six.  
By Brill-Noether theory any smooth curve of genus six $C$ has a special divisor $D$  with $\deg(D)=6$ and $h^0(C,D)=3$.
Let $\varphi_D$ be the morphism asociated to $D$:
$$\varphi_D: C\lra \bbP^2. $$
The curve $C$ will be called \emph{special} if it is either hyperelliptic, trigonal, bi-elliptic or isomorphic to a smooth plane quintic curve. The following is given for example in section A, Ch.V in \cite{ACGH}.
\begin{proposition} {\it Let $C$ be a smooth curve of genus six, then one of the followings holds$:$
  \begin{itemize}
  \item[$a)$]  $\varphi_D$ is birational and $\varphi_D(C)$ is an irreducible plane sextic having only double points.
 \item[$b)$]  $C$ is special.
 \end{itemize}}
\end{proposition}

\noindent \emph{Case a{\rm )} } 
Assume first that $\varphi_D(C)$ is a plane sextic with $4$ nodes $p_1,\dots,p_4$ in general position.
The blowing up of $\bbP^2$ in these points is a quintic del Pezzo surface $Y$ and $C\in |-2K_Y|$. In fact, the embedding   $C\subset Y\subset \bbP^ 5$ is the canonical embedding of $C$ and $Y$ is the unique quintic Del Pezzo surface containing $C$
(see e.g. \cite{SB}).  Let $e_0$ be the class of the pull back of a line and let $e_i$ be the classes of exceptional divisors over the points $p_i$.
The surface $Y$ contains $10$ lines
$$e_i \ , \ e_0 - e_i -e_j, \hspace{0.5cm} 1\leq i < j \leq 4.$$ It is known that the group of automorphisms of the dual graph of the $10$ lines is isomorphic to $S_5$.
The surface $Y$ admits exactly five birational morphisms to $ \bbP^2$, called \emph{blowing down maps}, induced by the linear systems:
$$e_0,\ 2e_0-\sum_{i=1}^4 e_i+e_j,\hspace{0.5cm} j=1,\dots,4.$$
 Note that any such morphism maps $C$ to a plane sextic with $4$ nodes. In fact also the converse holds i.e. there is a one-to-one correspondence between the set of blowing down maps for $Y$ and the set $W_6^2(C)$ of $g^2_6$ on $C$.
In particular the generic curve of genus $6$ has exactly five $g_6^2$.
The automorphisms group of $Y$ acts on the blowing down classes, giving a  representation $Aut(Y)\to S_5,$ 
which is known to be an isomorphism. The stabilizer of a blowing down model $\phi$ 
is given by projectivities permuting the $4$ points $p_1,\dots,p_4\in\bbP^2$ which are the image of the exceptional divisors of $\phi$, while an element of order five is realized by a quadratic transformation $\alpha$ with fundamental points at $p_1,p_2,p_3$  \cite[Theorem 10.2.2]{Do}.

If $p_1,\dots,p_4$ are not in general position then either $3$ of them lie on a line or two of them are infinitely near.
Note that anything worse is not admitted since $\varphi_D(C)$ is irreducible with at most double points. 
The blowing up of $\bbP^ 2$ in these points is a \emph{nodal} del Pezzo surface, i.e. $-K_Y$ is nef and big (see \cite{DO}). Equivalently, the anti-canonical model of $Y$ has at most rational double points. 
In this case the properties of the embedding $C\subset Y$ still hold, in particular $Y$ containing $C$ is unique  (\cite[5.14]{AH}).
However, the surface $Y$ may have less than five blowing down classes, i.e. $C$ has less than five $g_2^6$.  \\

\noindent \emph{Case b{\rm )} } The following characterization holds:
\begin{proposition}\label{spec}
A curve of genus six $C$ is special if and only if $\dim W_6^2(C)>0$.
\end{proposition}
\proof We have seen that if $C$ is not special, then $\dim W_6^2(C)=0$ and contains at most five points. We now see what happens for special curves (\cite{ACGH}).
\begin{enumerate}[$\bullet$]
\item If $C$ is \emph{trigonal} then it has two types of $g_6^2$: 
$$D=2g_3^1\hspace{0.4cm} \mbox{ and }\hspace{0.4cm} D(p)=K_C-g_3^1-p,\ \ p\in C.$$
 Hence $W_6^2(C)$ is one dimensional and has two irreducible components.
The plane model $\varphi_D(C)$ is a triple conic and $\varphi_{D(p)}(C)$ is a plane sextic with a triple point and a node.

\item If $C$ is isomorphic to a \emph{plane quintic} then any $g_6^2$ on $C$ is of type: $D(p)=g_5^2+p$, $p\in C$.
Hence $W_6^2(C)\cong  C$. The plane model $\varphi_{D(p)}(C)$ is a plane quintic.

\item If $C$ is \emph{bi-elliptic} i.e. there exists $\pi:C\rightarrow E$, where $E$ is an elliptic curve, then any $g_6^2$ corresponds to $\phi\circ\pi$ where $\phi$ is a $g_2^1$ on $E$.
The plane model of $C$ is a double cubic.

\item If $C$ is \emph{hyperelliptic} then any $g_6^2$ is of type:
$$ D(p,q)=K_C-g_2^1-p-q,\ \ p,q\in C.$$
Hence $W_6^2(C)\cong Sym^2(C)$. In fact $D=K_C-2g_2^1$ is a singular point of $W_6^2(C)$. The  plane model $\varphi_{D}(C)$ is a double rational cubic and 
$\varphi_{D(p,q)}(C) $ is  a double conic.
\end{enumerate}
\qed

\begin{remark}\label{sigma}
It follows that the moduli space of curves of genus six is birational to the GIT moduli space of plane sextics with 4 nodes up to the action of the group generated by projectivities and by  the birational transformation $\alpha$.

\end{remark}

\section{$K3$ surfaces associated to curves of genus 6}\label{}
\subsection{The geometric construction}\label{geo}
Let $C\subset \bbP^5$ be  the canonical model of a non-special smooth curve of genus six. By the remarks in the previous section, there is a  unique nodal Del Pezzo surface $Y$ such that $C$ lies in the anti-canonical model of $Y$ in $\bbP^5$. Let $Y'\to Y$ be the canonical resolution of rational double points of $Y$.  
Since $C \in |-2K_{Y'}|$, there exists a double cover
$$\pi : X \lra Y'$$
branched along $C$ and $X$ is a $K3$ surface.  
It is well known that $H^2(X,\bbZ)$, together with the cup product, is an
even unimodular  lattice of signature $(3,19)$.  
The covering involution $\sigma$ of $\pi$ acts on this lattice with eigenspaces
$$H^2(X,\bbZ)^{\pm} = \{ x \in H^2(X, \bbZ) : \sigma^*(x) = \pm x\}.$$

\begin{lemma}\label{ret}
$H^2(X,\bbZ)^+ \cong A_1(-1) \oplus A_1^4,\ H^2(X,\bbZ)^- \cong U\oplus U\oplus E_8\oplus A_1^{5}.$
 \end{lemma}
\proof 
By definition, the lattices $H^{\pm}=H^2(X,\bbZ)^{\pm}$ are 2-elementary, i.e. their discriminant groups are 
2-elementary abelian groups.
By \cite[Theorem 3.6.2]{N1} the isomorphism class of a $2$-elementary even indefinite lattice $L$ is determined uniquely by the triple $(s,\ell, \delta)$, where $s$ is the signature, $\ell$ is the minimal number of generators of $A_L$ and $\delta$ is $0$ (resp. $1$)  if the quadratic form on $A_L$ always assumes integer values (resp. otherwise).  
On the other hand \cite[Theorem 4.2.2]{N2} shows that $H^+$ has $s =(1,4),\ \ell=5,\ \delta =1$. 
Since $H^-$ is the orthogonal complement of $H^+$ in the unimodular lattice $H^2(X,\bbZ)$, it has $s=(2,15),\ \ell=5,\ \delta=1$.
Hence it is enough to check that the lattices in the right hand sides have the same triple of invariants.
\qed\\

\noindent Let $S_X$ be the Picard lattice of $X$ and let $T_X$ be its transcendental lattice:
$$S_X=H^2(X,\bbZ)\cap \omega_X^{\perp},\hspace{0.7cm} T_X=S_X^{\perp}.$$
Note that the invariant lattice $H^2(X,\bbZ)^+$ coincides with the pull-back of the Picard lattice of $Y$, hence
$$H^2(X,\bbZ)^+ \subset S_X, \hspace{0.7cm} T_X \subset H^2(X,\bbZ)^-.$$
If $\omega_X$ is a nowhere vanishing holomorphic 2-form on $X$, then $\omega_X\in T_X\otimes \bbC$, hence $\sigma^*(\omega_X) = -\omega_X$.

\begin{lemma}\label{-2}
There are no $(-2)$-vectors in $S_X \cap H^2(X,\bbZ)^-$.
\end{lemma}
\begin{proof}
Assume that $r$ is such a vector. By Riemann-Roch theorem we may assume that $r$ is effective.  Then 
$\sigma^*(r) = -r$ is also effective.  This is a contradiction.
\end{proof}

\subsection{Lattices}
We will denote by $L_{K3}$ an even unimodular lattice of signature $(3,19)$. This is known to be unique up to isomorphisms (see e.g. \cite[Theorem 1.1.1]{N1}), hence
the lattice $H^2(X, \bbZ)$ is isomorphic to $L_{K3}$.  Let $\bar e_0, \bar e_1,\dots,\bar e_4$ be the pull-backs of the classes $e_0, e_1, \dots, e_4$ under $\pi^*$.  These generate a sublattice of $S_X$ isometric to $A_1(-1) \oplus A_1^4.$
Let $$S =A_1(-1) \oplus A_1^4, \hspace{0.5cm}  T=U\oplus U\oplus E_8\oplus A_1^{ 5}.$$  
Denote by $s_0, s_1,\dots,s_4$ an orthogonal basis for $S$ with $s_0^2=2$, $s_i^2=-2$, $\ i=1,\dots, 4$ 
and denote by $r_1,\dots,r_5$ an orthogonal basis for the $A_1^{ 5}$ component of $T$.
\begin{lemma}\label{disc}
Let $\xi_i=r_i/2$, then the discriminant group $A_T$  consists of the following  vectors{\rm :}

$$\begin{array}{lll}

q(x) =0: & 0,\ \sum_{i\not=j}\xi_i,  \hspace{0.5cm} 1\leq j\leq  5\\
 
q(x) = 1: & \xi_i+\xi_j , \hspace{0.5cm} 1\leq i < j \leq 5\\
 
q(x) = -1/2: & \xi_i ,\ \hspace{0.5cm} 1\leq i \leq 5, \ \sum_{i=1}^5\xi_i\\
 
q(x) = -3/2: &  \sum_{i\not=j,k}\xi_i , \hspace{0.5cm} 1\leq j < k\leq 5.\\
 
\end{array}$$

\end{lemma}
\noindent It follows from  \cite[Theorem 1.14.4]{N1} that
$S$ can be embedded uniquely in $L_{K3}$ and $T$ is isomorphic to its orthogonal complement. Since $L_{K3}$ is unimodular, 
$$A_S \cong A_T \cong \bbF_2^5, \hspace{0.4cm} q_S \cong -q_T$$
and an isomorphism from $A_S$ to $A_T$ is given by
$$s_0/2\mapsto \xi_1,\hspace{0.5cm} (2s_0-\sum_{i=1}^4 s_i+s_j)/2 \mapsto \xi_{j+1},\hspace{0.5cm} j=1,\dots,4. $$
\begin{lemma}\label{iso} There are isomorphisms $\O(q_S) \cong \O(q_T) \cong S_5$ and the natural maps
$$\O(T) \to \O(q_T),\hspace{0.4cm} \O(S)\to \O(q_S)$$
are surjective. 
\end{lemma}
\proof The first statement follows from \cite{MS}. Note that $\O(q_T)$ acts on $A_T$ by permuting the $\xi_i$'s.
The surjectivity statement for $T$ is obvious, since clearly exist isometries of $T$ permuting the $r_i$'s.  
On the other hand, the automorphism group $S_5$ of $Y$ acts on $S$ as isometries.
These isometries act on $A_S$ as $S_5$.  More concretely, the isometries of $S$ permuting the $s_i$'s $(1\leq i \leq 4)$ 
and the isometry
$$s_0\mapsto 2s_0-s_1-s_2 - s_3,\ s_1\mapsto s_0- s_1- s_3, \ 
s_2\mapsto s_4, \ s_3\mapsto s_0- s_2- s_3,\ s_4\mapsto s_0- s_1- s_2$$
generate $\O(q_S)$.
\qed\\

\noindent In the following we will consider three arithmetic groups acting on $T$:
$$
\Gamma=\O(T),\hspace{0.4cm} \Gamma^{*}=\{\gamma\in \O(T):\ \gamma(\xi_1)=\xi_1\},
\hspace{0.4cm} \Gamma^{**}=\{\gamma\in \O(T):\ \gamma|A_T=1\}.$$
Note that $\Gamma/\Gamma^{**}\cong \O(q_T) \cong S_5$.
\begin{lemma}\label{su} Let $\O_T=\{\gamma \in \O(L_{K3}): \gamma(T)=T\}$.  Then the restriction homomorphisms 
$$\O_T\to \Gamma,\hspace{0.4cm} \{\gamma\in \O_T: \gamma(s_0)=s_0\}\to \Gamma^*
\mbox{ and }\hspace{0.4cm} \{\gamma\in \O_T: \gamma|S=1_S\}\to \Gamma^{**}$$ 
are surjective.
\end{lemma}
\proof Let $\gamma\in \Gamma$. By Lemma \ref{iso} there exists $\beta\in \O(S)$ such that $\beta=\gamma$ on $A_S\cong A_T$. Then the isometry $\beta\oplus \gamma$ on $S\oplus T$ lifts to an isometry in $\O_T$. 
If $\gamma\in \Gamma^*$ or $\Gamma^{**}$ then $\beta$ can be choosen such that $\beta(s_0)=s_0$ 
or $\beta=1_S$, respectively (see the proof of Lemma \ref{iso}).\qed
\begin{remark}\label{}
There are two orbits of vectors with $q(x)=-1/2$ under the 
action of $\O(q_T)$:
$$O_1=\{\sum_{i=1}^5 \xi_i\},\hspace{1cm} O_2=\{\xi_1, \dots,\xi_5\}.$$
\end{remark}

\subsection{Moduli spaces}
\noindent Since both $S$ and $T$ are $2$-elementary lattices, the isometry $(1_S, -1_T)$ on $S\oplus T$ can be extended to an isometry $\iota$ of $L_{K3}$. Let $\alpha : H^2(X,\bbZ) \to L_{K3}$ be an isometry satisfying 
$\alpha(H^2(X,\bbZ)^{+}) = S$.  Then $\iota \circ \alpha = \alpha\circ \sigma^*$. Since $\sigma^*(\omega_X) = -\omega_X$ then the \emph{period} 
$$p_X(\alpha)=\alpha_{\bbC}(\omega_X)$$ belongs to the set
$$\mathcal D = \{ \omega \in \bbP(T\otimes \bbC) : (\omega, \omega) = 
0, (\omega, \bar{\omega}) > 0  \},$$
called the \emph{period domain of $S$-polarized K3 surfaces}. 
By Lemma 2, there are no $(-2)$-vectors orthogonal to the period, hence $p_X(\alpha)$  belongs to the complement of the divisor
$$\mathcal H = 
\bigcup_{r \in T,\ r^2=-2}{\mathcal H_r}\ \mbox{ where }\  \mathcal H_r = \{ \omega \in \mathcal D : (r, \omega) = 0\}.$$
Consider the orbit spaces
$$\mathcal M=\mathcal D/\Gamma, \hspace{1cm} \mathcal M^*=\mathcal D/\Gamma^*, \hspace{1cm} \mathcal M^{**}=\mathcal D/\Gamma^{**}.$$   
Let  $\mathcal W_6^2$ be the moduli space of pairs $(C,D)$ where $C$ is a  smooth curve of genus $6$ and $D\in W_6^2(C)$  (see \cite{ACGH})
 and let $\mathcal {\widetilde M}_6$ be the moduli space of plane sextics with four ordered nodes.
\begin{theorem}\label{bir} 
The geometric construction in {\rm \ref{geo}} defines a birational map
$$\mathcal P^{**}: \mathcal {\widetilde M}_6\dashrightarrow \mathcal M^{**}.$$
The map $\mathcal P^{**}$ is equivariant for the natural action of $S_5$,  taking quotients for this action and for the action of a subgroup isomorphic to $S_4$ gives  birational maps
$$ \mathcal P:\mathcal M_6\dashrightarrow \mathcal M, \hspace{1cm} \mathcal P^*: \mathcal W_6^2\dashrightarrow \mathcal M^*.$$
In fact it induces an isomorphism
$$\mathcal M_6 \setminus \{ \mbox{special curves} \} \cong \mathcal M \setminus (\mathcal H/\Gamma).$$
 \end{theorem}
\proof 
Let $C$ be a plane sextic with $4$ ordered nodes.
The construction in \ref{geo} associates to $C$ a K3 surface $X$ which is birational to the the double cover of $\bbP^2$ branched along the plane sextic. If $C$ is general, then $S_X=H^2(X,\bbZ)^+$ is the pull-back of the Picard lattice of $Y$ and $\{ \bar e_0, \bar e_1,\dots, \bar e_4\}$ gives an ordered basis of $S_X$. 

In general, by using Lemma \ref{su}, choose a marking $\alpha:H^2(X,\bbZ)\rightarrow L_{K3}$ such that $\alpha(H^2(X,\bbZ)^+)\subset S$ and 
$\alpha (\bar e_i)=s_i $,\  $0\leq i\leq 4$.  By Lemma \ref{-2} $\alpha_{\bbC}(\omega_X) \in \mathcal D\setminus \mathcal H$. 
Moreover, if $\alpha_1, \alpha_2$ are two markings of this type, then $\alpha_2\alpha_1^{-1}$ preserves the ordered basis $\{ s_i\}$, hence its restriction to $T$ belongs to  $\Gamma^{**}$.
Thus we can associate to $C$ a point in $\mathcal D/\Gamma^{**}$, i.e. we defined a rational map
$\mathcal P^{**}:\mathcal {\tilde M}_6\dashrightarrow \mathcal M^{**}.$

Conversely, let $\omega\in \mathcal D\setminus \mathcal H$. By the surjectivity theorem of the period map (\cite{Ku,PP}) there exists a marked K3 surface $(X,\alpha)$ such that $ \alpha_{\bbC}(\omega_X)=\omega$. Then $\iota(\omega)=-\omega$ and there exist no 
$(-2)$-vectors in $T\cap \omega^{\perp}$ since $\omega\not\in \mathcal H$,  hence $\iota$ preserves an ample class.
It now follows from the Torelli theorem \cite[Theorem 3.10]{Na} that $\iota$ is induced by an automorphism $\sigma$ on $X$.

By \cite[Theorem 4.2.2]{N2} the fixed locus of $\sigma$ is a smooth curve $C$ of genus six. 
The quotient surface $Y=X/(\sigma)$ is smooth and the image of $C$ belongs to $|-2K_Y|$. Hence $-K_Y$ is nef and big with 
$K_Y^2=5$, i.e. $Y$ is a nodal quintic del Pezzo surface. 
In fact, the pull back of $\Pic(Y)$ is exactly $\alpha^{-1}(S)\subset S_X$.

If we choose $\omega\in \mathcal D\setminus \mathcal H$ up to the action of $\Gamma^{**}$ then, by Lemma \ref{su}, we get $\alpha$ up to an isometry in $\O_T$ which preserves an ordered basis $\{ s_i\}$.  Hence this gives a K3 surface $X$ with a class $\alpha^{-1}(s_i)\in\alpha^{-1}(S)$, \ $0\leq i \leq 4$.  By blowing down the corresponding four $(-1)$-curves on $Y$, we get a plane sextic with four ordered nodes.  This proves that $\mathcal P^{**}$ is birational.

The quotients $\Gamma/\Gamma^{**} \cong S_5$ and $\Gamma^*/\Gamma^{**} \cong S_4$ act on $\mathcal M^{**}$ 
and by Lemma \ref{su} they lift to isometries of $L_{K3}$ which preserve $T$ and $S$.
By taking the quotient for these actions, we get birational maps $\mathcal P$ and $\mathcal P^*$.
\qed

\section{The discriminant divisor}
In the previous section we introduced a divisor $\mathcal H$ in $\mathcal D$. The image of this divisor in $\calM$ or $\calM^*$ will be called \emph{discriminant divisor}. We now describe its structure and its geometric meaning.

\subsection{Irreducible components}
\begin{lemma}\label{orbit} Let $\Delta$ be the set of vectors $r\in T$ with $r^2=-2$, then
\begin{enumerate}[$\bullet$]
\item the group $\Gamma$ has three orbits in $\Delta :$
$$\Delta_1 = \{ r \in \Delta : r/2 \notin T^*\},\ \ \Delta_2 = \{ r \in \Delta : r/2 \in O_1\},\ \ 
\Delta_3 = \{ r \in \Delta : r/2 \in O_2\};$$

\item the group $\Gamma^*$ has $4$ orbits in $\Delta :$\ \ 
$\Delta_1, \Delta_2$ and two orbits decomposing $\Delta_3$
$$\Delta_{3a}=\{r\in \Delta: r/2=\xi_2\},\ \ \Delta_{3b}=\{r\in \Delta: r/2=\xi_1\}.$$
\end{enumerate}
\end{lemma}
\proof
Given a vector $r\in \Delta$ we will classify the embeddings of $\Lambda=\langle r\rangle$ in $T$ up to the action of $\Gamma$ by applying \cite[Proposition 1.15.1]{N1}. 
We first need to give an isometry $\alpha$ between a subgroup of $A_{\Lambda}$ and a subgroup of $A_T \cong \bbF_2^5$. If $H$ is such a subgroup, then either $H=0$ or $H=\bbF_2$. Note that $H=\bbF_2$ if and only if $r/2\in T^*$. 

In case $H=0$, since there is a unique a lattice $K$ with $q_{ K}=q_{\Lambda}\oplus (-q_T)$  and $\O(K)\rightarrow \O(q_K)$ is surjective by \cite[Theorem 1.14.2]{N1}, then by \cite[Proposition 1.15.1]{N1} there is a unique embedding of $\Lambda$ in $T$ such that $\Lambda\oplus \Lambda^{\perp}=T$.

In case $H=\bbF_2$ there are two different embeddings of $\Lambda$, according to
the choice of $\alpha(r/2)$ in $O_1$ or $O_2$. This gives the first assertion.  

The second assertion can be proved in a similar way, by observing that $\Gamma^*$ has three orbits on the set of vectors $x\in A_T$ with $q(x)=-1/2$.
\qed\\

  For $r \in \Delta_i$, let $ T_i = \{ x \in T : (x, r)= 0\}$ and
denote by $S_i$ the orthogonal complement of $ T_i$ in $L_{K3}$.  Then we have:
\begin{lemma}\label{}
$$\begin{array}{ll}
 S_1 \cong A_1(-1)\oplus A_1^{5},& T_1 \cong U \oplus U \oplus E_7 \oplus A_1^{ 5},\\
 S_2 \cong U(2) \oplus D_4,& T_ 2 \cong U\oplus U(2)\oplus E_8\oplus D_4,\\
 S_3 \cong U \oplus A_1^{4},& T_3 \cong U \oplus U \oplus E_8 \oplus A_1^{4}.
\end{array}$$
\end{lemma}
\proof
Because of Lemma \ref{orbit} the isomorphism class of $T_i$ does not depend on the choice of $r\in \Delta_i$. If $r\in\Delta_1$ or $\Delta_3$ then we can assume $r$ to be one generator of $E_8$ or respectively one generator of $A_1$ in a decomposition
$T=U\oplus U\oplus E_8\oplus A_1^{ 5}$.
If $r\in\Delta_2$ we can assume $r$ to be a generator of $A_1$ in a decomposition
$T= U\oplus U(2) \oplus E_8 \oplus D_4\oplus A_1$.
In all these cases the orthogonal complement of $r$ in $T$ can be easily computed.
The lattices $S_i$ can be computed by applying \cite[Theorem 3.6.2]{N1}.
\qed
 \begin{corollary}
The divisor $\mathcal H/\Gamma$  has $3$ irreducible components $\mathcal H_1,\mathcal H_2,\mathcal H_3$ and
$\mathcal H/\Gamma^*$ has $4$ irreducible components  $\mathcal H^*_1, \mathcal H^*_2, \mathcal H^*_{3a}, \mathcal H^*_{3b}$ such that
\begin{enumerate}[$\bullet$]
\item $\mathcal H^*_i\rightarrow \mathcal H_i,\ \ i=1,2 $
have degree $5$,
\item $\mathcal H^*_{3a}\rightarrow \mathcal H_3,\ \   \mathcal H^*_{3b}\rightarrow \mathcal H_3$ have degree $4$ and $1$ respectively.
\end{enumerate}
\end{corollary}
\noindent Let $\iota_i$ be the isometry of $L_{K3}$ defined by $\iota_i | S_i = 1_{S_i}$ and $\iota_i | T_i = -1_{T_i}$.
The following can be proved by means of Torelli theorem, as in the proof of Theorem \ref{bir}.
\begin{lemma}\label{aut}
There exists a K3 surface $X_i$ such that $S_{X_i}\cong S_i$ and carrying an involution $\sigma_i$ of $X_i$ with $\sigma_i^*=\iota_i$.
\end{lemma}

\subsection{Curves of genus six with a node}
 Let $C_1$ be a generic plane sextic with five nodes. The blowing up of the projective plane at the nodes is a quartic del Pezzo
surface $Y_1$ and its double cover branched along the strict transform of $C_1$ is a K3 surface $X$. 
Alternatively, if we blow up the plane at four nodes, we get a quintic del Pezzo surface on which the strict transform of $C_1$ is a curve 
of genus six with a node.  The pull-back of $\Pic(Y_1)$ is a sublattice of the Picard lattice of $X$ isomorphic to $S_1$.
We now show that also the converse is true
\begin{proposition}\label{h1}\
The K3 surface $X_1$ is birational to the double cover of a quintic del Pezzo surface branched along a generic curve of genus six 
with a node or, equivalently, to a double plane branched along a generic sextic with $5$ nodes. 
\end{proposition}
\proof
Consider the involution $\sigma_1$ on $X_1$ as in Lemma \ref{aut}. By  \cite[Theorem 4.2.2]{N2} the  fixed locus of $\sigma_1$ is a smooth curve $C_1$ of genus $5$. The quotient of $X_1$ by $\sigma_1$ is a smooth rational surface $Y_1$ and the image of $C_1$ belongs to $|-2K_{Y_1}|$, hence $Y_1$ is a del Pezzo surface of degree $4$.

Any $(-1)$-curve $e$ on $Y_1$ intersects the image of $C_1$ at two points since $(-K_{Y_1},e)=1$.  
Hence, contracting one $(-1)$-curve of $Y_1$ we get a quintic del Pezzo surface where the image of $C_1$ is a curve of genus six 
with a node, and contracting five disjoint $(-1)$-curves $C_1$ is mapped to a plane sextic with five nodes.   
 \qed
 
\begin{corollary}\label{h1c}
The divisor $\mathcal H_1$ is birational to the moduli space of curves of genus six with one node and 
$\mathcal H^*_1$  to the moduli space of plane sextics with $5$ nodes, with one marked. 
\end{corollary}
\proof  Taking the quotient of $\mathcal H_1$ for the action of $\Gamma$, we identify two markings on $X_1$ which give the same embedding of $\alpha^{-1}(S)$ in $\Pic(X_1)$. This data identifies a $(-1)$-curve on $Y_1$, whose contraction gives a quintic Del Pezzo surface and a curve of genus 6 with a node.
The group $\Gamma^*$, instead, identifies two markings on $X_1$ if they also give the same embedding of $\alpha^{-1}(h)$ in the Picard lattice. This class gives a blowing down map on $Y_1$ with a distinguished exceptional divisor. 

Using these remarks and Proposition \ref{h1}, the result follows as in the proof of Theorem \ref{bir}.\qed
 \subsection{Plane quintics}
Let $C_2$ be a smooth plane quintic and let $L$ be a line transversal to $C_2$. The minimal resolution of the double plane branched along $C_2\cup L$ is a K3 surface $X$. The Picard lattice of $X$ contains  five 
disjoint $(-2)$-curves, coming from the resolution of singularities, and a $(-2)$-curve which is the proper transform of $L$. These rational curves generate a lattice which is isomorphic to $S_2$.

\begin{proposition}\label{h2}\
The surface $X_2$ is birational to a double plane branched along the union of a plane quintic and a line. 
\end{proposition}
\proof
This was proved in \cite[Ch.6]{L}. \qed
\begin{corollary}\label{quintic}
 The divisor $\mathcal H_2$ is birational to the moduli space of pairs $(C,L)$ where $C$ is a plane quintic and $L$ is a line, while
$\mathcal H^*_2$ parametrizes triples $(C,L,p)$ where $p\in C\cap L$.\end{corollary}
\proof 
The first statement is \cite[Corollary 6.21]{L}. 
The second statement can be proved similarly to Corollary \ref{h1c}. \qed
\subsection{Trigonal curves of genus six}
Let $C\subset \bbP^5$ be the canonical model of a trigonal curve of genus $6$. Any $3$ points in the $g_1^3$ lie on a line by Riemann-Roch theorem and the closure of the union of all these lines is a quadric $Q$ 
such that the curve $C$ belongs to $|4f+3e|$, where $e,f$ are the rulings of $Q$.
The minimal resolution of the double cover of $Q$ branched along the union of $C$ with a line $L\in |e|$  is a K3 surface $X$.
The ruling $f$, the proper transform of $L$ and the exceptional divisors over  the four points in $C\cap L$ generate a sublattice of $S_X$ isomorphic to $S_3$.

As before, we now prove a converse statement.

 \begin{proposition}\label{h3}\
The surface $X_3$ is birational to:
\begin{enumerate}[$\bullet$]
\item the double cover of a quadric $Q$ branched along a line and a trigonal curve of genus six.
\item the double cover of a Hirzebruch surface $\bbF_4$ branched along a curve with $4$ nodes in $|3h|$ and the rational curve in $|s|$, where $h^2=4,\ s^2=-4,\ (h,s)=0$.
\end{enumerate}
\end{proposition}
\proof
By \cite[Theorem 4.2.2]{N2} the set of fixed points of $\sigma_3$ on $X_3$ is the disjoint union of a smooth curve $C$ of genus $6$ and a smooth rational curve $L$.
Since $S_3\subset  \Pic(X_3)$, $X_3$ admits an elliptic fibration $\pi$ with a section and four singular fibers of Kodaira type $I_2$
or type $III$. 
Since any fiber of $\pi$ is preserved by $\sigma_3$, then $L$ is a section of $\pi$ and $C$ intersects each fiber in $3$ points.
Hence $C$ has a triple cover to $\bbP^ 1$ and its ramification points are the singular points of irreducible fibers of $\pi$.

We will denote by $F_1,\dots, F_4$ the singular fibers of $\pi$ of type $I_2$ or $III$, by $E_i$ the component of $F_i$ meeting $L$ and by $E'_i$ the other component. 
Let $p:X_3\rightarrow Y_3$ be the quotient by the involution $\sigma_3$. Note that $p(E_i)$ and $p(E'_i)$ are $(-1)$-curves.

By contracting the curves $p(E_i)$, we get a smooth quadric surface. This gives the first assertion.

On the other hand, contracting the curves $p(E'_i)$, we get a Hirzebruch surface $\bbF_4$ (note that the image of $L$ has self-intersection $-4$). Since $C$ intersects the ruling in $3$ points, each $E'_i$ at two points and it does not intersect $L$, then its image in $\bbF_4$ has $4$ nodes and belongs to the class $3h$. This gives the second assertion.
\qed

 \begin{corollary}\ \label{trig}
\begin{enumerate}[$\bullet$]
\item  The divisor $\mathcal H_3$ is birational to the moduli space of pairs $(C,L)$ where $C$ is a trigonal curve of genus $6$ and $L\in |K_C-2g_3^1|$.
\item The divisor $\mathcal H^*_{3a}$ parametrizes pairs $(C,p)$ where $C$ is trigonal and $p\in C$ or, equivalently, plane sextics with a node and a triple point.
\item The divisor $\mathcal H^*_{3b}$ is birational to the moduli space of curves in $|3h|$ of $\bbF_4$ with $4$ nodes. 
\end{enumerate}
\end{corollary}
\proof 
The first statement follows from Proposition \ref{h3} and the remarks at the beginning of this subsection since, by adjunction formula, the restriction of  $e$ to $C$ coincides with $K_{C}-2g_3^1$.

Given a  trigonal curve $C\subset Q$ of genus six and $p\in C$, there exists a unique line $L\in |e|$ through $p$. 
This determines a K3 surface $X$ with $S_3\subset S_X$ as before.
Moreover, the projection of $C$ from $p$ is a plane sextic with a triple point and a double point. The hyperplane class of $\bbP^2$ 
induces the linear system $K_{C}-g_3^1-p$ on $C$ and its pull-back 
to $X$ is a nef class $h$ with $h^2=2$.

Conversely, a generic point in $\mathcal H^*_{3a}\cup \mathcal H^*_{3b}$ gives a K3 surface $X$ with $S_X\cong S_3=U\oplus A_1^{4}$ and a degree two polarization $h$. 
Let $e,f$ be a basis of $U$ and $e_1,\dots,e_4$ an orthogonal basis of $A_1^{4}$. Up to an isometry of $S_3$ we can assume that $r=e-f$ and that $f$ gives an elliptic fibration on $X$. 
The orthogonal complement $S_3\cap r^{\perp}\cong S$ has two types of degree two polarizations: $h_j=2(e+f)-\sum_{i=1}^4 e_i+e_j$ or $h=e+f$.

A point in $\mathcal H^*_{3b}$ gives a polarization $h_b$ such that $h_b/2=r/2$ in $A_{T}\cong A_S$, hence $h_b=h$. The class $h_b$ contains $r$ in the base locus and  $2h_b$ maps $X$ onto a cone over a rational normal quartic. In fact, the morphism associated to $2h_b$ is exactly the contraction of the curves $p(E'_i)$ and the image of the curve $L$ described in the proof of Proposition \ref{h3}.

A point in $\mathcal H^*_{3a}$ gives a polarization $h_a=h_j$ for some $j=1,\dots,4$. In this case $h_a$ has no base locus and gives a generically 2:1  map  $X\rightarrow \bbP^2$. 
The branch locus of this map is a plane sextic with a triple point (in the image of $r$) and a node (in the image of $e_j$). The line through the two singular points intersects the sextic in one more point $p$. Hence this gives a pair $(C,p)$, where $C$ is trigonal and $p\in C$.\qed

\begin{remark}
The two irreducible components in  $\mathcal M^*$ over $\mathcal H_3$    correspond to the components in $\mathcal W_6^2$  over the trigonal divisor in $\mathcal M_6$. With the notation in the proof of Proposition \ref{spec}: the divisor $\mathcal H^*_{3a}$ corresponds to pairs $(C, D(p))$
and $\mathcal H^*_{3b}$ to $(C,2g_3^1)$. 
This agrees with \cite{Sh}, where it is proved that the triple conic, which is the plane model of $C$ associated to $2g_3^1$ (Proposition \ref{spec}), ``represents'' K3 surfaces with a degree two polarization with a fixed component.
 \end{remark}

\begin{remark}
Let $C$ be a plane sextic with four nodes $p_1,\dots,p_4$ such that $p_1,p_2,p_3$ lie on a line $L$. The blowing up of the plane in these points is a nodal del Pezzo surface $Y$ (see section 1) and the double cover of $Y$ branched along the proper transform of $C$ is a K3 surface $X$. 
The pencil of lines through $p_4$ induces an elliptic fibration on $X$ with general fiber $f$, $3$ fibers of type $I_2$ and
two sections $s_1, s_2$, given by the two (disjoint) inverse images of the line $L$.
In particular, the Picard lattice of $X$ contains the sublattice
$S'=U\oplus A_1^{ 3}\oplus <-4>$, 
where $U$ is generated by the fiber $f$ and $s_1$, $A_1^{ 3}$ by the reducible components in each fiber and $<-4>$ by $2f+s_1-s_2$.

Conversely, let $r\in T$ be a primitive vector with $r^2=-4$ such that
 $r/2\in A_T$, then its orthogonal complement in $T$ is isomorphic to
$T'=U\oplus U\oplus E_8\oplus A_1^{ 3 }\oplus <-4>$ and $T'^{\perp}\cong S'$.

By choosing a different blow-down map for $Y$ we get a plane sextic with a tacnode and two nodes. In fact, the elliptic fibration described above is induced by the pencil of lines through the tacnode.
\end{remark}

\section{Compactifications}
\subsection{Satake-Baily-Borel compactification}
The  moduli spaces $ \mathcal M$, $\mathcal M^*$  are quasi-projective algebraic varieties. Since they are arithmetic quotients of a symmetric bounded domain, we can consider their Satake-Baily-Borel (SBB) compactifications $\overline{\mathcal M}$ and $\overline{\mathcal M^*}$ (see \cite{BB} and \cite{Sc},\S\,2). 
  
It is known that boundary components of the SBB compactification are in bijection with primitive isotropic sublattices of $T$ up to $\Gamma$ and $\Gamma^*$ respectively, such that $k$-dimensional boundary components correspond to rank $k+1$ isotropic sublattices. Since $T$ has signature $(2,15)$, the boundary components will be either $0$ or $1$ dimensional.

\begin{lemma}\label{iso1}
Let $\mathcal I$ be the set of primitive isotropic vectors in $ T$. There are two orbits in $\mathcal I$ with respect to the action of $\Gamma:$
$$\mathcal I_1=\{v \in \mathcal I:\  (v,T)=\mathbb Z\}\hspace{0.5in} 
\mathcal I_2=\{v\in \mathcal I:\ (v,T)=2\mathbb Z\} .$$
There are three orbits with respect to $\Gamma^* :$ $\mathcal I_1$ and  two orbits decomposing $\mathcal I_2$.
\end{lemma}
\proof
By Proposition 4.1.3 in \cite{Sc} there is a bijection between orbits of isotropic vectors in $T$ modulo $\Gamma$ ($\Gamma^*$) and isotropic vectors in $A_T$ modulo the induced action of $\Gamma$ ($\Gamma^*$).
By Lemma \ref{iso} the map $\Gamma\rightarrow \O(q_T)$ is surjective and clearly the image of $\Gamma^*$ is given by elements of $\O(q_T)$ fixing $\xi_1$.
Then it follows from Lemma \ref{disc} that there are exactly two  orbits of isotropic vectors in $A_T$ for the action of $\Gamma$ and three for the action induced by $\Gamma^*$.  
\qed

\begin{corollary}
The boundaries of $\overline{\mathcal M}$ and $\overline{\mathcal M}^*$ contain two  
and three zero-dimensional components respectively. 
\end{corollary}
We will denote by $p,q$ the zero-dimensional boundary components of $\overline{\mathcal M}$ corresponding to the orbits $\mathcal I_1, \mathcal I_2$ in Lemma \ref{iso1} respectively and with  $q_1, q_2$ the zero-dimensional boundary components of $\overline{\mathcal M}^*$ corresponding to the orbits of $\Gamma^*$ decomposing $\mathcal I_2$.

\begin{remark}\label{t}
By \cite[Theorem 3.6.2]{N1} there is also an isomorphism
$$T\cong U\oplus U(2)\oplus A_1\oplus D_4\oplus E_8.$$
In the following we will denote by $e,f$ and $e',f'$ the standard bases of $U$ and $U(2)$, by $\beta$ a generator of $A_1$, by $\gamma_1,\dots,\gamma_4$ and $\alpha_1,\dots,\alpha_8$ the standard root bases of $D_4$  and $E_8$.
Note that $e,f\in \mathcal I_1$ and $e',f'\in \mathcal I_2$.
\end{remark}

We now classify one dimensional boundary components in $\overline{\mathcal M}$ by studying $\Gamma$-orbits of primitive isotropic planes in $T$. We will say that such a plane is of \emph{type} $(i,j)$, $i,j=1,2$ if it is generated by a vector in $\mathcal I_i$ and one in $\mathcal I_j$.

Let $\mathcal G_1$ be the genus of $E_8\oplus A_1^{ 5}
$ and let $\mathcal G_2$ be the genus of $E_8\oplus A_1\oplus D_4$.
If $N$ is a lattice in  $\mathcal G_1$, then $T\cong U\oplus U\oplus N$ by \cite[Theorem 3.6.2]{N1}.
By taking two isotropic vectors, each in one copy of $U$, we get an isotropic plane in $T$ of type $(1,1)$.
Similarly, if $N_2\in \mathcal G_2$ then $T\cong U\oplus U(2)\oplus N_2$ and the plane generated by a generator of $U$ and one of $U(2)$  is isotropic of type $(1,2)$.

\begin{lemma}
The isomorphism classes of lattices in $\mathcal G_1$ and $\mathcal G_2$ are given in the following table. 
\end{lemma}
\begin{table}[ht]
\newcommand\T{\rule{0pt}{2.8ex}}
\newcommand\B{\rule[-1.6ex]{0pt}{0pt}}
\begin{tabular}{l|l|l|l }
 & $R$ &  $\mathcal G_1 $ & $ \mathcal G_2 $ \B \\ 
\hline
a& $E_8^3$ & $E_7\oplus D_4\oplus A_1^2$, $D_6^2\oplus A_1$, $E_8\oplus A_1^5$ &  $A_1\oplus E_8\oplus D_4$ \T \B\\  
\hline 
 b& $E_7^2\oplus D_{10}$ & $\overline{E_7\oplus A_1^6}, \overline{D_6\oplus D_4\oplus A_1^3}, D_8\oplus D_4\oplus A_1, 
\overline{D_8\oplus A_1^5}, D_{10}\oplus A_1^3$ & $E_7\oplus D_6, \overline{D_{10}\oplus A_1^3}$  \T \B \\  
\hline
c& $ D_{16}\oplus E_8$& $\overline{D_8\oplus A_1^5}$ &  $D_{12}\oplus A_1$  \T \B \\  
\hline
 d& $A_{17}\oplus E_7$& $(A_1^4)^{\perp}$ in $A_{17}$ &  \T \B \\ 
 
\end{tabular}
\ \\
\ \\
\ \\
\caption{One dimensional boundary components}\label{table}
\end{table}

\proof
The orthogonal complements of $E_8\oplus A_1^{ 5}$ and $E_8\oplus A_1\oplus D_4$ in $E_8^{ 3}$ are isomorphic to $R_1=E_7\oplus A_1^4$  and $R_2=E_7\oplus D_4$ respectively.
By Proposition 6.1.1, \cite{Sc} the isomorphism classes in $\mathcal G_1$ and $\mathcal G_2$ can be obtained by taking the orthogonal complements of primitive embeddings of $R_1$  
and respectively $R_2$ into even negative definite unimodular lattices of rank $24$, i.e. Niemeier lattices. 
These lattices are uniquely determined by their root sublattice $R$, hence they are denoted by $N(R)$ (see \cite{CS}, Chap. 18). 
In order to determine all lattices in the $\mathcal G_i$ we first classify all primitive embeddings of $R_1, R_2$ into $R$ and take their orthogonal complements $R_i^{\perp}$ in $R$.  
Then we take the primitive overlattice $\overline{R_i^{\perp}}$ of 
$R_i^{\perp}$ in $N(R)$ which contains $R_i^{\perp}$ as a subgroup of index at most $2$.  Here we have used the classification of embeddings between root lattices due to Nishiyama \cite{Ni}.  This gives isomorphism classes $\overline{R_i^{\perp}}$ in $\mathcal G_i$.
In Table \ref{table} all root lattices $R$ appear such that $R_i$ can be embedded in $N(R)$ and the corresponding lattices in 
$\mathcal G_1$ and $\mathcal G_2$.  If $R_i^{\perp}$ is primitive
in $N(R)$, then we omit the overline.

\qed

\begin{theorem}\label{1bou}
The boundary of $\overline{\mathcal M}$ contains 
$14$ one dimensional components $B_1,\dots,B_{14}$ where the closure of $B_i$, $i=1,\dots,10$ contains only $p$ and the closure of $B_j$, $j=11,\dots, 14$ contains both $p$ and $q$.
\end{theorem}
\proof
As remarked before, to the lattices in $\mathcal G_1$ we can associate  isotropic planes of type $(1,1)$ in $T$ which are not $\Gamma$-equivalent. Conversely, by Lemma 5.2 in \cite{Sc}, any isotropic plane $E$ of type $(1,1)$ can be embedded in $U\oplus U$ and $T\cong U\oplus U\oplus E^{\perp}/E$ where $E^{\perp}/E\in \mathcal G_1$. Hence, boundary components containing only $p$ are in one-to-one correspondence with lattices in $\mathcal G_1$. 

The proof is more subtle for isotropic planes of type $(1,2)$. Note that if $v\in T$ is a primitive isotropic vector of type $2$ and $E$ is an isotropic plane containing $v$, then $E$ determines a primitive vector in $M_v=v^{\perp}/\mathbb Z v$. Hence, isotropic planes of type $(1,2)$ correspond to orbits of isotropic vectors in $M_v$.
In this case $M_v\cong U\oplus E_8\oplus D_4\oplus A_1$ and orbits of isotropic vectors can be determined by Vinberg's algorithm (see \S 1.4 \cite{V} or \S 4.3 \cite{St}).

By \cite[Theorem 0.2.3]{N2}, the Weyl group $W(M_v)$ has finite index in $O(M_v)$. This implies that the algorithm will finish in a finite number of steps.
To start the algorithm we fix the vector $\bar x=e+f$. Then at each step we have to choose roots $x\in M_v$ such that the \emph{height} $$h=\frac{(x,\bar x)}{\sqrt{-x^2}}$$
is minimal and $(x_i,x_j)\geq 0$ for $j=1,\dots,i-1$. In our case we get:\\

\noindent $i)$ $(x,\bar x)=0$:\hspace{0.4cm}  $u:=e-f,\ \alpha_1,\dots,\alpha_8, \gamma_1,\dots, \gamma_4,\beta.$\\
$ ii)$ $(x,\bar x)=1 $:\hspace{0.4cm}  $\alpha:=f+\bar\alpha_8,\ \gamma:=f+\bar\gamma_1,\ \beta':=f-\beta.$\\
$iii)$ $(x,\bar x)=4 $:\hspace{0.4cm}  $\delta_i:=2(e+f)-\beta+\bar \alpha_1+\bar\gamma_i,\ \ j=2,3,4.$\\
$iv)$ $(x,\bar x)=12 $: \hspace{0.4cm}  $\alpha ':=6(e+f)-3\beta+2\bar\alpha_4+\bar\gamma_2+\bar\gamma_ 3+\bar\gamma_4$\\

\noindent where $\bar \alpha_1,\dots,\bar \alpha_8$ and $\bar \gamma_1,\dots,\bar\gamma_4$ are the dual bases of $E_8$ and $D_4$.
We now draw the Dyinkin diagram associated to these roots. Let $g_{ij}=(e_i,e_j)/\sqrt{e_i^2e_j^2}.$
Then two vertices $i,j$ corresponding to vectors $e_i, e_j$ are connected by
$$\xymatrix @R=0.05in@C=0.3in{
\bullet\ar@{}[r]&\bullet&\ \mbox{ if }& g_{ij}=0,\\
\bullet\ar@{-}[r]&\bullet&\ \mbox{ if }& g_{ij}=1/2,\\
\bullet\ar@{--}[r]&\bullet&\ \mbox{ if }& g_{ij}=1,\\
\bullet\ar@{-}[r]|\|&\bullet&\ \mbox{ if }& g_{ij}>1
.}$$
The diagram in our case is given in Figure 1 (see also Figure 5, \cite{Kon}).
Note that the symmetry group of the diagram is 
$\mathbb Z_2\times S_3$
and it can be easily seen that all symmetries can be realized by isometries in $\Gamma$. The maximal parabolic subdiagrams of rank $13$ are of four types :
\begin{figure}
$\xymatrix @R=0.3in@C=0.3in{
& & & &\bullet\ar@{--}[rr]^<{\gamma_2} & &\bullet\ar@{-}[dr]\ar@{}[r]^<{\delta_2}& & & & \\
& & \bullet\ar@{-}[r]\ar@{-}[dd]^<{\gamma}& \bullet\ar@{-}[ur]\ar@{-}[r]^<{\gamma_1}\ar@{-}[dr]& \bullet\ar@{--}[rr]^<{\gamma_3}  & &\bullet\ar@{-}[r]^<{\delta_3}&\bullet\ar@{-}[r]^<{\alpha_1}&\bullet\ar@{-}[dd]^<{\alpha_2}\\
& & & & \bullet\ar@{--}[rr]^<{\gamma_4} & &\bullet\ar@{-}[ur]^<{\delta_4}& & &\\
& & \bullet\ar@{-}[d]\ar@{-}[r]^<{u}&\bullet\ar@{--}[r]^<{\beta'}&\bullet\ar@{--}[uuurr] \ar@{--}[uurr] \ar@{--}[urr]^<{\beta}\ar@{-}[rr]|\| & &\bullet\ar@{--}[uuull]\ar@{--}[uull]\ar@{--}[ull]\ar@{--}[r]^<{\alpha'} &\bullet\ar@{-}[r]^<{\alpha_4} &\bullet\ar@{-}[d]^<{\alpha_3}\\
& & \bullet\ar@{-}[r]^<{\alpha}& \bullet\ar@{-}[rr]^<{\alpha_8}& &\bullet \ar@{-}[rr]^<{\alpha_7}& &\bullet \ar@{-}[r]^<{\alpha_6}& \bullet\ar@{}[r]^<{\alpha_5} &
}$\\
\ \\
 \caption{The Dynkin diagram of $W(M_v)$} \hspace{3cm}
\end{figure}

$$\tilde E_8\oplus \tilde D_4\oplus \tilde A_1=\langle \alpha_i,\alpha,\beta',\beta,\gamma,\gamma_j\rangle\ \ \ i=1,\dots,8;\ j=1,\dots,4$$
$$\tilde D_{12}\oplus \tilde A_1=\langle  \alpha_i, \alpha, u, \gamma, \gamma_j, \beta, \delta_4\rangle\ \ \ i=2,\dots,8;\ j= 2,3 $$
$$\tilde E_7\oplus \tilde D_6=\langle \alpha_i, \delta_2, \alpha,u, \beta', \gamma, \gamma_j\rangle\ \ \ i=1,\dots,7;\ j=3,4.$$
$$\tilde D_{10}\oplus \tilde A_1^3=\langle \alpha_i,\alpha,u,\gamma, \beta', \gamma_j, \delta_j  \rangle\ \ \ i=2,\dots,8;\ j=2,3,4.$$
\noindent Note that each type is an orbit for the action of $\Gamma$.  These subdiagrams correspond to non-equivalent isotropic vectors in $M_v$. Hence, we get $4$ isotropic planes in $T$ containing a vector in $\mathcal I_2$ and a direct analysis shows that all of them are of type $(1,2)$. \qed\\

It follows from  the proof of Theorem \ref{1bou} that the boundary components of $\overline{\mathcal M}$ are in one-to-one correspondence with the lattices in $\mathcal G_1$ and $\mathcal G_2$. These lattices
 appear in connection to degenerations of K3 surfaces as explained for example in \cite{Sc}. 
 This allows to compare the SBB compactification with more geometrically meaningful compactifications,  as the ones obtained by means of geometric invariant theory.

In case of K3 surfaces with a degree two polarization this is well-understood (\cite{Sh}, \cite{F}, \cite{Lo2}). 
 Table \ref{table2} describes the correspondence between type II boundary components of the GIT compactification of plane sextics and one dimensional boundary components of the Baily-Borel compactification for degree two K3 surfaces.
\begin{table}[ht]
\newcommand\T{\rule{0pt}{2.8ex}}
\newcommand\B{\rule[-1.6ex]{0pt}{0pt}}
\begin{tabular}{c c|c}
& GIT & SBB\\
\hline
IIa:& $(x_0x_2+a_1x_1^2)(x_0x_2+a_2x_1^2)(x_0x_2+a_3x_1^2)=0$& $E_8\oplus E_8\oplus A_1$ \T \B \\
\hline
IIb:& $x_2^2f_4(x_0,x_1)=0.$& $E_7\oplus D_{10}$ \T \B\\
\hline
IIc:& $(x_0x_2+x_1^2)^2f_2(x_0,x_1,x_2)=0.$& $D_{16}\oplus A_1$ \T \B\\
\hline
IId:& $f_3(x_0,x_1,x_2)^2=0.$ & $A_{17}$ \T \\
\end{tabular}
\ \\
\ \\
\ \\
\caption{GIT and SBB of plane sextics}\label{table2}
\end{table}
The lattice appearing in the SBB column is $E^{\perp}/E$, where $E$ is the isotropic lattice associated to the boundary component. 
 \begin{remark} In the proof of Theorem \ref{1bou} we showed that boundary components of $\overline{\mathcal M}$ containing only $p$ in their closure correspond to primitive embeddings of the lattice $E_7\oplus A_1^{4}$ into Neimeier lattices.
Equivalently, they correspond to primitive embeddings of the lattice $A_1^{4}$ in the root lattices $E_8\oplus E_8$, $E_7\oplus D_{10}$, $D_{16}$, $A_{17}$. 
Note that a double cover branched over a node has an $A_1$ singularity hence, embedding $A_1^4$ in the root lattices is equivalent to choose a distribution of the $4$ nodes on the corresponding configurations in Table \ref{table2} (where more than one node can ``collapse'' to the same singular point of the configuration).

For example, let $q_1,q_2$ be the two singular points in the IIa configuration. We can either embed one node in $q_1$ and $3$ nodes in $q_2$ (this gives the root lattice $E_7\oplus D_1\oplus A_1$), two nodes in $q_1$ and two in $q_2$ (this gives the root lattice $D^2_6 \oplus A_1$) or $4$ nodes in $q_1$ (this gives the root lattice $E_8\oplus A_1^5$).

Similarly, boundary components containing both $p$ and $q$ in their closure correspond to embeddings of the lattice $D_{4}$ into the previous root lattices. Note that a double cover branched over a triple point has a $D_4$ singularity.

In fact we conjecture that a one dimensional boundary component $B$ of $\overline{\mathcal M}$ of type a, b, c or d (see Table \ref{table}) corresponds to 
a boundary component of type IIa, IIb, IIc or IId respectively   with 
\begin{enumerate}[$\bullet$]
\item $4$ marked nodes (eventually collapsing) if $q\not\in B$ 
\item a marked triple point if $q\in B$.
\end{enumerate}

 Note that the configuration IId has no triple points, in fact there is no one-dimensional boundary component of type $d$ containing $q$ in its closure.
\end{remark}

\begin{remark}
 By corollaries \ref{quintic} and \ref{trig} the moduli space $\mathcal M$ contains two divisors which are birational to $\bbP^2$ and $\bbP^1$ fibrations over the locus of plane quintics and trigonal curves respectively. This suggests that we need to blow-up the moduli space of curves of genus six in order to extend the period map to these loci.

Bi-elliptic and hyperelliptic curves of genus six are mapped to one dimensional boundary components of $\overline{\mathcal M}$.
In fact, the configuration IIc is a plane model for hyperelliptic curves and case IId is the plane model of a bi-elliptic curve of genus six (see \S 1).

\end{remark}

\end{document}